\documentclass[twoside,a4paper,reqno,11pt]{amsart} 
\usepackage[top=28mm,right=28mm,bottom=28mm,left=28mm]{geometry}
\usepackage{microtype}
\usepackage{amsfonts, amsmath, amssymb, amsbsy, mathrsfs, bm, latexsym, stmaryrd, hyperref, array, enumitem, textcomp, url}
\usepackage[table]{xcolor}

\renewcommand{\a}{\alpha}
\renewcommand{\b}{\beta}

\newcommand{\e}{\epsilon}

\newcommand{\s}{\sigma}
\renewcommand{\O}{\Omega}

\newcommand{\leqs}{\leqslant}
\newcommand{\geqs}{\geqslant}

\newcommand{\vs}{\vspace{3mm}}

\makeatletter
\newcommand{\imod}[1]{\allowbreak\mkern4mu({\operator@font mod}\,\,#1)}
\makeatother

\theoremstyle{plain}
\newtheorem{theorem}{Theorem} 
 
\newtheorem{corol}[theorem]{Corollary}
\newtheorem{thm}{Theorem}[section] 
\newtheorem{lem}[thm]{Lemma}
\newtheorem{prop}[thm]{Proposition}

\newtheorem*{theorem*}{Theorem} 
\newtheorem*{conj*}{Conjecture}

\theoremstyle{definition}
\newtheorem{rem}[thm]{Remark}
\newtheorem{defn}[thm]{Definition}
\newtheorem*{def-non}{Definition} 

\newtheorem{remk}{Remark}

\begin{document}

\title[On the involution fixity of simple groups]{On the involution fixity of simple groups}

\author{Timothy C. Burness}
\address{T.C. Burness, School of Mathematics, University of Bristol, Bristol BS8 1UG, UK}
\email{t.burness@bristol.ac.uk}

\author{Elisa Covato}
\address{E. Covato, Bristol, UK}
\email{elisa.covato@gmail.com}

\date{\today} 

\begin{abstract}
Let $G$ be a finite permutation group of degree $n$ and let ${\rm ifix}(G)$ be the involution fixity of $G$, which is the maximum number of fixed points of an involution. In this paper we study the involution fixity of almost simple primitive groups whose socle $T$ is an alternating or sporadic group; our main result classifies the groups of this form with ${\rm ifix}(T) \leqs n^{4/9}$. This builds on earlier work of Burness and Thomas, who studied the case where $T$ is an exceptional group of Lie type, and it strengthens the bound ${\rm ifix}(T) > n^{1/6}$ (with prescribed exceptions), which was proved by Liebeck and Shalev in 2015. A similar result for classical groups will be established in a sequel. 
\end{abstract}

\maketitle

\section{Introduction}\label{s:intro}

Let $G \leqs {\rm Sym}(\O)$ be a permutation group on a finite set $\O$. Let ${\rm fix}(g)$ be the number of elements in $\O$ fixed by $g \in G$ and set 
\[
{\rm fpr}(g,\O) = \frac{{\rm fix}(g)}{|\O|},
\]
which is called the \emph{fixed point ratio} of $g$. This is a classical concept in permutation group theory and bounds on fixed point ratios find a wide range of applications, especially in the context of primitive groups. For instance, we refer the reader to the recent survey article \cite{Bur} for a discussion of some powerful  applications concerning bases for permutation groups, the random generation of simple groups and the structure of monodromy groups of coverings of the Riemann sphere.

In this paper we study ${\rm fix}(g)$ in the setting where $G$ is an almost simple primitive permutation group and $g \in G$ is an involution. We call 
\[
{\rm ifix}(G) = \max\{{\rm fix}(g) \,:\, \mbox{$g \in G$ is an involution}\}
\]
the \emph{involution fixity} of $G$ and we are interested in comparing ${\rm ifix}(G)$ with the degree of $G$. This is closely related to the more general concept of \emph{fixity}, which is defined to be the maximal number of points fixed by a non-identity element. The latter notion was originally introduced by Ronse \cite{Ronse} in 1980 and there are more recent papers by Liebeck, Saxl and Shalev \cite{LSh,SS} on the fixity of primitive groups (also see \cite{MW}, where the transitive groups with fixity at most $2$ are studied). Let us also  highlight work of Bender \cite{Bender} from the early 1970s, which determines the finite transitive groups $G$ with ${\rm ifix}(G) = 1$.

Our main motivation stems from \cite{LSh}, where Liebeck and Shalev use the O'Nan-Scott theorem to investigate the structure of the primitive groups of  degree $n$ with fixity at most $n^{1/6}$. Their main result for an almost simple group $G$ with socle $T$ shows that ${\rm ifix}(T) > n^{1/6}$, with specified exceptions (see \cite[Theorem 4]{LSh}). With a view towards applications, it is desirable to strengthen this lower bound (at the expense of some additional exceptions). The first step in this direction was taken by Burness and Thomas in \cite{BThomas}, where the almost simple groups with socle an  
exceptional group of Lie type $T$ and ${\rm ifix}(T) \leqs n^{4/9}$ are determined. In this paper, we extend the analysis in \cite{BThomas} to the almost simple groups with socle an alternating or sporadic group. The remaining classical groups will be handled in a sequel, which will complete our study of involution fixity for almost simple primitive groups.

Our main result is the following. In the statement, 
$\mathcal{S}$ denotes the set of finite simple groups that are either alternating or sporadic. 

\begin{theorem}\label{t:main}
Let $G \leqs {\rm Sym}(\O)$ be an almost simple primitive permutation group of degree $n$ with socle $T \in \mathcal{S}$ and point stabilizer $H$. Set $H_0 = H \cap T$. Then one of the following holds: 
\begin{itemize}\addtolength{\itemsep}{0.2\baselineskip}
\item[{\rm (i)}] ${\rm ifix}(T) > n^{4/9}$.
\item[{\rm (ii)}] $H_0$ has odd order and ${\rm ifix}(T) = 0$.
\item[{\rm (iii)}] $(T,n) = (A_5,5)$ and ${\rm ifix}(T) =1$. 
\item[{\rm (iv)}] $n^{\a} \leqs {\rm ifix}(T) \leqs n^{4/9}$ and $(T,H_0,{\rm ifix}(T),n,\a)$ is recorded in Table \ref{tab:main}.
\end{itemize}
\end{theorem}

{\small 
\begin{table}
\begin{center}
\[
\begin{array}{llllll}\hline
T & H_0 & {\rm ifix}(T) &  n & \a & \mbox{Conditions} \\ \hline
A_5 & S_3 & 2 & 10 & 0.301 & \\
& D_{10} & 2 & 6 & 0.386 & \\
A_6 & 3^2{:}4 & 2 & 10 & 0.301 & \\
& A_5 & 2 & 6 & 0.386 & \mbox{$G = A_6$ or $S_6$} \\
& D_{10} & 4 & 36 & 0.386 & \mbox{$G = {\rm M}_{10}$, ${\rm PGL}_{2}(9)$ or $A_6.2^2$} \\
& S_4 & 3 & 15 & 0.405 & \mbox{$G = A_6$ or $S_6$} \\
& D_8 & 5 & 45 & 0.422 & \mbox{$G = {\rm M}_{10}$, ${\rm PGL}_{2}(9)$ or $A_6.2^2$} \\
A_7 & {\rm L}_{2}(7) & 3 & 15 & 0.405 & G = A_7 \\
A_9 & 3^2{:}{\rm SL}_{2}(3) & 8 & 840 & 0.308 & \\
& {\rm L}_{2}(8){:}3 & 8 & 120 & 0.434 & G = A_9  \\
A_{10} & {\rm M}_{10} & 24 & 2520 & 0.405 & \\
A_{11} & {\rm M}_{11} & 24 & 2520 & 0.405 & G = A_{11} \\
{\rm J}_{1} & 2^3{:}7{:}3 & 5 & 1045 & 0.231 & \\
& 11{:}10 & 12 & 1596 & 0.336 & \\
& 7{:}6 & 20 & 4180 & 0.359 & \\
& 19{:}6 & 20 & 1540 & 0.408 & \\
& {\rm L}_{2}(11) &  10 & 266 & 0.412 & \\ 
{\rm J}_{2} & A_5 & 60 & 10080 & 0.444 & \\
{\rm J}_{3} & 2^4{:}(3 \times A_5) & 50 & 17442 & 0.400 & \\
& 2^{2+4}{:}(3 \times S_3) & 85 & 43605 & 0.415 & \\
& 3^2.3^{1+2}{:}8 & 80 & 25840 & 0.431 & \\
{\rm McL} & 3^{1+4}{:}2S_5 & 56 & 15400 & 0.417 &  \\
{\rm He} & 7^2{:}2.{\rm L}_{2}(7) & 64 & 244800 & 0.335 & \\
{\rm O'N} & 3^4{:}2^{1+4}D_{10} & 1064 & 17778376 & 0.417 & \\
{\rm Co}_{1} & 5^2{:}2A_5 & 3244032 & 1385925602181120 & 0.430 & \\
{\rm HN} & {\rm U}_{3}(8){:}3 & 800 & 16500000 & 0.402 & \\
{\rm Th} & 2^5.{\rm L}_{5}(2) & 2169 & 283599225 & 0.394 & \\
& 7^2{:}(3 \times 2S_4) & 645120 & 12860819712000 & 0.443 & \\ \hline
\end{array}
\]
\caption{The cases with $n^{\a} \leqs {\rm ifix}(T) \leqs n^{4/9}$ in Theorem \ref{t:main}(iv)}
\label{tab:main}
\end{center}
\end{table}}

\begin{remk}\label{r:main}
Let us make some comments on the statement of Theorem \ref{t:main}.
\begin{itemize}\addtolength{\itemsep}{0.2\baselineskip}
\item[{\rm (a)}] The groups arising in part (ii) with $|H_0|$ odd are determined in \cite[Theorem 2]{LSa} (also see \cite[Lemma 2.1]{LSh}). The possibilities are as follows:
\[
\begin{array}{lll} \hline
T & H_0 & \mbox{Conditions} \\ \hline
A_p & {\rm AGL}_{1}(p) \cap T & \mbox{$p$ prime, $p \equiv 3 \imod{4}$} \\
& & \mbox{$G = S_p$ if $p=7,11,23$} \\
\mbox{${\rm J}_{3}$, ${\rm O'N}$} & \mbox{$19{:}9$, $31{:}15$ (resp.)} &  G = T.2 \\
\mbox{${\rm M}_{23}$, ${\rm Th}$, $\mathbb{B}$} & \mbox{$23{:}11$, $31{:}15$, $47{:}23$ (resp.)} & \\ \hline
\end{array}
\]
\item[{\rm (b)}] The number $\a$ recorded in the fifth column of Table \ref{tab:main} is equal to $\log {\rm ifix}(T) / \log n$, expressed to $3$ significant figures.
\item[{\rm (c)}] The theorem reveals that there are only finitely many groups of the given form with $1 \leqs {\rm ifix}(T) \leqs  n^{4/9}$. However, it is straightforward to show that there are infinitely many with $1 \leqs {\rm ifix}(T) \leqs n^{1/2}$. For example, we can take $T = A_p$ and $H = {\rm AGL}_{1}(p) \cap G$, where $p$ is any prime with $p \equiv 1 \imod{4}$ (see Remark \ref{r:aff2}).
\item[{\rm (d)}] Theorem \ref{t:main} already has an application in finite geometry. Indeed, we refer the reader to \cite[Section 6]{BPP}, where the result is used to study point-primitive generalized quadrangles. 
\end{itemize}
\end{remk}

By combining Theorem \ref{t:main} with \cite[Theorem 1]{BThomas}, we get the following corollary.

\begin{corol}\label{c:main}
Let $G \leqs {\rm Sym}(\O)$ be an almost simple primitive permutation group of degree $n$ with socle $T$ and point stabilizer $H$. Set $H_0 = H \cap T$ and assume $|H_0|$ is even and $T$ is not isomorphic to a classical group. Then one of the following holds: 
\begin{itemize}\addtolength{\itemsep}{0.2\baselineskip}
\item[{\rm (i)}] ${\rm ifix}(T) > n^{1/3}$.
\item[{\rm (ii)}] $(T,n) = ({}^2B_2(q),q^2+1)$ and ${\rm ifix}(T) =1$. 
\item[{\rm (iii)}] $(T,H_0,{\rm ifix}(T),n)=(A_9, 3^2{:}{\rm SL}_{2}(3), 8, 840)$ or $({\rm J}_{1}, 2^3{:}7^3{:}3, 5, 1045)$.
\end{itemize}
\end{corol}

The proof of Theorem \ref{t:main} is presented in Sections \ref{s:alt} and \ref{s:spor}, where we handle the groups with an alternating and sporadic socle, respectively. We freely employ computational methods, using \textsf{GAP} \cite{GAP} and {\sc Magma} \cite{magma}, when it is feasible to do so. In particular, the argument for sporadic groups in Section \ref{s:spor} makes extensive use of the character tables (and associated fusion maps) that are available in the \textsf{GAP} Character Table Library \cite{GAPCTL}. As one might expect, the O'Nan-Scott theorem provides a framework for our proof when the socle $T$ is an alternating group. Indeed, this key result divides the possibilities for the point stabilizer $H$ into several families and we proceed by considering each family in turn. 

The notation we use in this paper is fairly standard. We will write $C_n$, or just $n$, for a cyclic group of order $n$ and $G^n$ denotes the direct product of $n$ copies of $G$. An unspecified extension of $G$ by a group $H$ will be denoted by $G.H$; if the extension splits then we write $G{:}H$. We adopt the standard notation for simple groups of Lie type from \cite{KL}, which differs slightly from the notation in \cite{Atlas}. All logarithms are in the natural base, unless stated otherwise.

\section{Symmetric and alternating groups}\label{s:alt}

Let $G \leqs {\rm Sym}(\O)$ be an almost simple primitive permutation group of degree $n$ with socle $T$ and point stabilizer $H$. Set $H_0 = H \cap T$ and note that $H$ is a maximal subgroup of $G$ such that $G = HT$. Then $n = |T:H_0|$ and  
\begin{equation}\label{e:fix}
{\rm fix}(t) = \frac{|t^T \cap H_0|}{|t^T|} \cdot n
\end{equation}
for all $t \in T$, where $t^T$ denotes the conjugacy class of $t$ in $T$. We will adopt this notation for the remainder of the paper.

In this section, we prove Theorem \ref{t:main} for the groups with socle $T=A_m$. Recall that if $t \in T$ is an involution with cycle-shape $(2^k,1^{m-2k})$, then
\[
|t^T| = \frac{m!}{k!(m-2k)!2^k}.
\]
We begin by handling the groups with $m \leqs 25$.

\begin{prop}\label{p:alt}
The conclusion to Theorem \ref{t:main} holds if $m \leqs 25$. 
\end{prop}

\begin{proof}
This is a straightforward {\sc Magma} \cite{magma} computation. First assume $G = A_m$ or $S_m$. Working in the natural permutation representation of degree $m$, we use the function \texttt{MaximalSubgroups} to construct a set of representatives of the conjugacy classes of maximal subgroups $H$ of $G$. Given an involution $t \in T$, we can then compute $|t^T \cap H_0|$ and $|t^T|$, which gives ${\rm fix}(t)$ via \eqref{e:fix}. We then obtain ${\rm ifix}(T)$ by  taking the maximum over a set of representatives of the conjugacy classes of  involutions in $T$ and the desired result quickly follows. Finally, if $T = A_6$ and $G$ is one of ${\rm PGL}_{2}(9)$, ${\rm M}_{10}$ or $A_6.2^2$ then we can proceed in an entirely similar manner, working with a permutation representation of $G$ of degree $10$.
\end{proof}

For the remainder of this section, we may assume $G = A_m$ or $S_m$ with  $m > 25$. Our aim is to establish the bound ${\rm ifix}(T)>n^{4/9}$. 

The possibilities for $H$ are described by the O'Nan-Scott theorem (see \cite{LPS}, for example), which divides the maximal subgroups of $G$ into the following families (in parts (e) and (f), $S$ denotes a non-abelian finite simple group):
\begin{itemize}\addtolength{\itemsep}{0.2\baselineskip}
\item[{\rm (a)}] \emph{Intransitive}: $H = (S_k \times S_{m-k}) \cap G$, $1 \leqs k < m/2$.
\item[{\rm (b)}] \emph{Imprimitive}: $H = (S_k \wr S_r) \cap G$, $m=kr$, $1 < k < m$.
\item[{\rm (c)}] \emph{Affine}: $H = {\rm AGL}_{d}(p) \cap G$, $m = p^d$, $p$ prime, $d \geqs 1$. 
\item[{\rm (d)}] \emph{Product-type}: $H = (S_k \wr S_r) \cap G$, $m=k^r$, $k \geqs 5$, $r \geqs 2$.
\item[{\rm (e)}] \emph{Diagonal-type}: $H = (S^k.({\rm Out}(S) \times S_k)) \cap G$, $m = |S|^{k-1}$, $k \geqs 2$. 
\item[{\rm (f)}] \emph{Almost simple}: $S \leqs H \leqs {\rm Aut}(S)$, $m = |H:K|$ for some maximal subgroup $K$ of $H$.
\end{itemize}

We will consider each family of subgroups in turn. Before we begin the analysis of case (a), let us record some useful preliminary lemmas.

\begin{lem}\label{l:basic}
Suppose $|H_0|$ is even, $|H_0| \leqs |T|^{\a}$ and $|t^T| \leqs |T|^{\b}$ for every involution $t \in H_0$. Then ${\rm ifix}(T) > n^{4/9}$ if $5-5\a-9\b>0$.
\end{lem}

\begin{proof}
Let $t \in H_0$ be an involution. Then $|t^T \cap H_0| \geqs 1$ and $|t^T| \leqs |T|^{\b}$, so ${\rm fix}(t) \geqs n|T|^{-\b}$ and thus ${\rm ifix}(T) > n^{4/9}$ if $n>|T|^{9\b/5}$. The result now follows since $n = |T:H_0| \geqs |T|^{1-\a}$.
\end{proof}

\begin{lem}\label{l:inv}
If $T = A_m$ and $m > 20$, then $|t^T| < |T|^{11/20}$ for every involution $t \in T$.
\end{lem}

\begin{proof}
The groups with $m \leqs 54$ can be checked using {\sc Magma}, so let us assume $m \geqs 55$. Recall that if $G$ is a finite group and $\mathcal{I}(G)$ is the set of involutions in $G$, then $|\mathcal{I}(G)|^2 < k(G)\cdot |G|$, where $k(G)$ is the number of conjugacy classes of $G$ (see \cite[Chapter 4]{Isaacs}, for example). As a special case, we deduce that
\[
|\mathcal{I}(S_m)|^2 < m!p(m),
\]
where $p(m)$ is the partition function, and thus it suffices to show that
\begin{equation}\label{e:part}
2^{11}p(m)^{10} < m!.
\end{equation}
Indeed, if this inequality holds then $|\mathcal{I}(S_m)| < |T|^{11/20}$ and the  desired bound follows. 

By the main theorem of \cite{Prib} we have $p(m) < m^{-3/4}e^{c\sqrt{m}}$, where $c = \pi\sqrt{2/3}$, so \eqref{e:part} holds if $f(m)>1$, where
\[
f(m):=\frac{m^{15/2}m!}{2^{11}e^{10c\sqrt{m}}}.
\]
For $m \geqs 55$ we have 
\[
\frac{f(m+1)}{f(m)} = \frac{\left(1+\frac{1}{m}\right)^{15/2}(m+1)}{e^{10c(\sqrt{m+1}-\sqrt{m})}} \geqs \frac{m}{6} > 1,
\]
so $f$ is an increasing function and the result follows since $f(55)>1$. 
\end{proof}

\begin{lem}\label{l:cor}
Let $T = A_m$ with $m>20$. If $|H_0|$ is even and $|H_0|^{100}<|T|$, then ${\rm ifix}(T)>n^{4/9}$.
\end{lem}

\begin{proof}
This follows by combining Lemmas \ref{l:basic} and \ref{l:inv}. 
\end{proof}

\subsection{Intransitive subgroups}\label{ss:intran}

In this section we will assume $H = (S_k \times S_{m-k}) \cap G$ is a maximal intransitive subgroup of $G$, where $1 \leqs k < m/2$. We may identify $\O$ with the set of $k$-element subsets of $\{1, \ldots, m\}$. In particular, $n = \binom{m}{k}$.

\begin{prop}\label{p:intrans}
If $m \geqs 7$, then ${\rm ifix}(T) > n^{1/2}$.
\end{prop}

\begin{proof}
We claim that ${\rm fix}(t)>n^{1/2}$, where $t = (1,2)(3,4) \in T$. If $k = 1$ then $n=m$, ${\rm fix}(t) = m-4$ and the result follows. Now assume $k \geqs 2$. Clearly, $t$ fixes a $k$-set $\Gamma$ if and only if $\Gamma \cap \{1,2,3,4\}$ is either empty, or one of $\{1,2\}$, $\{3,4\}$ or $\{1,2,3,4\}$. Therefore,
\[
{\rm fix}(t) = \binom{m-4}{k} + 2\binom{m-4}{k-2} + \binom{m-4}{k-4}
\]
where the final term is $0$ if $k=2$ or $3$. The cases with $m<10$ can be checked directly, so let us assume $m \geqs 10$. We claim that
\begin{equation}\label{e:bd1}
\binom{m-4}{k} + 2\binom{m-4}{k-2} > \binom{m}{k}^{\frac{1}{2}},
\end{equation}
which implies that ${\rm fix}(t) > n^{1/2}$. 

To see this, we first express the binomial coefficients $\binom{m-4}{k}$ and $\binom{m-4}{k-2}$ in terms of $\binom{m}{k}$ and we deduce that it suffices to show that 
\[
\binom{m}{k}^{\frac{1}{2}}\left(\frac{f(k)g(k)}{m(m-1)(m-2)(m-3)}\right)>1,
\]
where $f(k) = (m-k)(m-k-1)$ and $g(k) = 2k(k-1)+(m-k-2)(m-k-3)$. Since $k \leqs \frac{1}{2}(m-1)$, we calculate that $f(k) \geqs \frac{1}{4}(m^2-1)$ and $g(k) \geqs \frac{2}{3}m^2-4m+\frac{21}{4}$. In addition, we have $\binom{m}{k} \geqs \binom{m}{2}$ and thus \eqref{e:bd1} holds if $h(m)>1$, where
\[
h(m) := \frac{\binom{m}{2}^{\frac{1}{2}}(m+1)\left(\frac{2}{3}m^2-4m+\frac{21}{4}\right)}{4m(m-2)(m-3)}.
\]
Now
\[
\frac{h(m+1)}{h(m)} = \left(\frac{m+1}{m-1}\right)^{1/2}\cdot \frac{h_1(m)}{h_2(m)}
\]
with
\[
h_1(m) = m(m+2)(m-3)\left(\frac{2}{3}m^2-\frac{8}{3}m+\frac{23}{12}\right) = h_2(m)+\frac{11}{2}m^2-\frac{41}{4}m+\frac{21}{4} > h_2(m),
\]
so $h(m)$ is an increasing function and the result follows since $h(10)>1$.
\end{proof}

\subsection{Imprimitive subgroups}\label{ss:imprim}

Next we turn to the imprimitive subgroups of the form $H = (S_k \wr S_r) \cap G$, where $m=kr$ and $1<k<m$. We identify $\O$ with the set of partitions of $\{1, \ldots, m\}$ into $r$ subsets of size $k$. Note that
\[
n = |\O| = \frac{(kr)!}{k!^rr!} =: f(k,r).
\]

\begin{prop}\label{p:imprim}
If $m \geqs 9$, then ${\rm ifix}(T) > n^{1/2}$.
\end{prop}

\begin{proof}
We claim that ${\rm fix}(t) > n^{1/2}$ for $t = (1,2)(3,4) \in T$. 

First assume $k=2$, so $r \geqs 5$. Clearly, $t$ stabilizes a partition in $\O$ if and only if the partition contains $\{1,2\}$ and $\{3,4\}$, or $\{1,3\}$ and $\{2,4\}$, or $\{1,4\}$ and $\{2,3\}$. Therefore, ${\rm fix}(t) = 3f(2,r-2)$ and it suffices to show that $g(r)>1$, where
\[
g(r):=\frac{9f(2,r-2)^2}{f(2,r)}.
\]
Now 
\[
\frac{g(r+1)}{g(r)} = \frac{(2r-3)^2}{2r+1}>1,
\]
so $g(r)$ is an increasing function and the result follows since $g(5)>1$.

Now assume $k \geqs 3$. A partition in $\O$ is fixed by $t$ if and only if it has a part containing $\{1,2\}$ and another containing $\{3,4\}$, or $k \geqs 4$ and it has a part containing $\{1,2,3,4\}$. Therefore,
\[
{\rm fix}(t) = \binom{m-4}{k-2}\binom{m-k-2}{k-2}f(k,r-2) + \binom{m-4}{k-4}f(k,r-1)
\]
and it suffices to show that 
\[
g(k,r):= \binom{kr-4}{k-2}^2\binom{kr-k-2}{k-2}^2\frac{f(k,r-2)^2}{f(k,r)} >1.
\]

We claim that if $k$ is fixed then $g(k,r)$ is increasing as a function of $r$. To see this, first observe that
\[
\frac{g(k,r+1)}{g(k,r)} = \frac{r}{(r-1)^2}\binom{kr+k-4}{k}\frac{(m-1)(m-2)(m-3)}{(m+k-1)(m+k-2)(m+k-3)}
\]
and we have the bounds
\[
\binom{kr+k-4}{k} \geqs \left(r+1-\frac{4}{k}\right)^k \geqs \left(r - \frac{1}{3}\right)^k
\]
and
\[
\frac{(m-1)(m-2)(m-3)}{(m+k-1)(m+k-2)(m+k-3)} \geqs \left(\frac{m-3}{m+k-3}\right)^3 \geqs \left(\frac{k}{6}+1\right)^{-3}
\]
since $k \geqs 3$ and $m \geqs 9$. It is routine to check that 
\[
\left(r - \frac{1}{3}\right)^k \geqs (r-1)\left(\frac{k}{6}+1\right)^{3}
\]
and this justifies the claim. 

Therefore, for $k \geqs 4$ we have 
\[
g(k,r) \geqs g(k,2) = \binom{2k-4}{k-2}^2\frac{1}{f(k,2)}
\]
and
\[
\frac{g(k+1,2)}{g(k,2)} = \frac{2(2k-3)^2(k+1)}{(k-1)^2(2k+1)} > 1,
\]
so $g(k,r) \geqs g(4,2) > 1$. Similarly, if $k=3$ then $r \geqs 3$ and $g(3,r) \geqs g(3,3)>1$. The result follows.
\end{proof}

\subsection{Affine subgroups}\label{ss:affine}

In this section we assume $H = {\rm AGL}_{d}(p) \cap G$ and $m = p^d$, where $p$ is a prime and $d \geqs 1$. Note that 
\[
n = |\O| \geqs \frac{|T|}{|{\rm AGL}_{d}(p)|} = \frac{(p^d-1)!}{2|{\rm GL}_{d}(p)|}.
\]

Write ${\rm AGL}_{d}(p) = V{:}L$, where $V = (\mathbb{F}_p)^d$ and $L = {\rm GL}(V)$. Now ${\rm AGL}_{d}(p)$ acts faithfully on $V$ by affine transformations $(v,x): u \mapsto v+u^x$ and this embeds ${\rm AGL}_{d}(p)$ in $S_m$. Note that if $t = (v,x) \in {\rm AGL}_{d}(p)$ then $t^2=1$ if and only if $v^x = -v$ and $x^2=1$. 

\begin{defn}\label{d:tk}
Fix a basis $\{e_1, \ldots, e_d\}$ for $V$. With respect to this basis, let us  define  $x_k = [-I_k,I_{d-k}]$ if $p \ne 2$ and $x_k = [A^k,I_{d-2k}]$ if $p=2$, where $A = \left(\begin{smallmatrix} 0 & 1 \\ 1 & 0 \end{smallmatrix}\right)$. In particular, if $p=2$ then $1 \leqs k \leqs d/2$ and $x_k$ is a block-diagonal matrix with $k$ blocks equal to $A$. For $k \geqs 1$ set $t_k = (v,x_k) \in {\rm AGL}_{d}(p)$, where $v=e_1$ if $p \ne 2$, otherwise $v = 0$. Note that $t$ is an involution.
\end{defn}

\begin{lem}\label{l:aff}
Let $t = t_k \in {\rm AGL}_{d}(p)$. Then $t$ has cycle-shape $(2^{p^{d-k}(p^k-1)/2},1^{p^{d-k}})$ as an element of $S_m$ and we have 
$|C_{{\rm AGL}_{d}(p)}(t)| = p^{d-k}|C_{{\rm GL}_{d}(p)}(x_k)|$.
\end{lem}

\begin{proof}
First consider the cycle-shape of $t$. Since $t$ is an involution, it suffices to show that it fixes exactly $p^{d-k}$ vectors in $V$. Suppose $w = \sum_{i}a_ie_i \in V$ is fixed by $t$. 

First assume $p \ne 2$. Here $w = w^t = w^{x_k}+e_1$ and thus
\[
\sum_{i=1}^{d}a_ie_i = (-a_1+1)e_1 + \sum_{i=2}^{k}(-a_i)e_i + \sum_{i=k+1}^da_ie_i, 
\]
so $a_1 = \frac{1}{2}$ and $a_i = 0$ for $2 \leqs i \leqs k$. There are no conditions on the coefficients $a_i$ for $i > k$, so $t$ fixes precisely $p^{d-k}$ vectors and the result follows. Similarly, if $p = 2$ then $w = w^t = w^{x_k}$ and
\[
\sum_{i=1}^{d}a_ie_i = \sum_{i=1}^{k}(a_{2i}e_{2i-1}+a_{2i-1}e_{2i}) + \sum_{i=2k+1}^{d}a_ie_i,
\]
which implies that $a_{2i-1}=a_{2i}$ for $1 \leqs i \leqs k$. Therefore $t$ fixes $2^{k}2^{d-2k}=2^{d-k}$ vectors as claimed.

Now let us consider the centralizer of $t$. Suppose $p \ne 2$ and $(u,y) \in {\rm AGL}_{d}(p)$. Then $(u,y)$ centralizes $t$ if and only if $y \in C_{{\rm GL}_{d}(p)}(x_k) = {\rm GL}_{k}(p) \times {\rm GL}_{d-k}(p)$ and $u+e_1^y = e_1+u^{x_k}$. Given $y \in C_{{\rm GL}_{d}(p)}(x_k)$, a straightforward calculation shows that there are $p^{d-k}$ vectors $u \in V$ such that 
$u+e_1^y = e_1+u^{x_k}$ and thus 
\[
|C_{{\rm AGL}_{d}(p)}(t)| = p^{d-k}|C_{{\rm GL}_{d}(p)}(x_k)| = p^{d-k}|{\rm GL}_{k}(p)||{\rm GL}_{d-k}(p)|.
\]
Similarly, if $p=2$ then $(u,y) \in {\rm AGL}_{d}(2)$ centralizes $t$ if and only if $y \in C_{{\rm GL}_{d}(2)}(x_k)$ and $u^{x_k} = u$. Since the $1$-eigenspace of $x_k$ on $V$ is $(d-k)$-dimensional, we get
\[
|C_{{\rm AGL}_{d}(2)}(t)| = 2^{d-k}|C_{{\rm GL}_{d}(2)}(x_k)| = 2^{d-k+2dk-3k^2}|{\rm GL}_{k}(2)||{\rm GL}_{d-2k}(2)|
\]
and the result follows.
\end{proof}

\begin{prop}\label{p:aff1}
If $d=1$ then one of the following holds:
\begin{itemize}\addtolength{\itemsep}{0.2\baselineskip}
\item[{\rm (i)}] $p \equiv 3 \imod{4}$ and ${\rm ifix}(T) = 0$.
\item[{\rm (ii)}] $p=5$, $n=6$ and ${\rm ifix}(T) = 2$.
\item[{\rm (iii)}] $p \equiv 1 \imod{4}$, $p \geqs 13$ and ${\rm ifix}(T)> n^{4/9}$.
\end{itemize}
\end{prop}

\begin{proof}
First observe that $H_0 = p{:}\frac{1}{2}(p-1)$ and $n = (p-2)!$. In particular, if $p \equiv 3 \imod{4}$ then $|H_0|$ is odd and thus ${\rm ifix}(T) = 0$ as claimed. Now assume $p \equiv 1 \imod{4}$. If $p=5$ then $H_0 = D_{10}$ and  ${\rm ifix}(T) = 2$, so let us assume $p \geqs 13$. Let $t \in H_0$ be an involution. By applying  Lemma \ref{l:aff}, noting that $H_0$ has a unique conjugacy class of involutions, we deduce that 
\[
|t^T \cap H_0| = p,\;\; |t^T| = \frac{p!}{2^{(p-1)/2}\left(\frac{1}{2}(p-1)\right)!}
\]
and thus \eqref{e:fix} gives
\begin{equation}\label{e:eq1}
{\rm ifix}(T) = {\rm fix}(t) = \frac{2^{(p-1)/2}\left(\frac{1}{2}(p-1)\right)!}{p-1}.
\end{equation}
It follows that ${\rm ifix}(T)>n^{4/9}$ if and only if $f(p)>1$, where
\[
f(p):=\frac{2^{(p-1)/2}\left(\frac{1}{2}(p-1)\right)!}{(p-1)(p-2)!^{4/9}}.
\]
The result now follows since $f(p+2) = (p-1)^{5/9}p^{-4/9}f(p) > f(p)$ and $f(13)>1$.
\end{proof}

\begin{rem}\label{r:aff2}
The proof of Proposition \ref{p:aff1} reveals that there are infinitely many groups $G$ as in Theorem \ref{t:main} with $1 \leqs {\rm ifix}(T) \leqs n^{1/2}$. Indeed, if we take $T = A_p$ and $H = {\rm AGL}_{1}(p) \cap G$, where $p$ is a prime such that $p \equiv 1 \imod{4}$, then ${\rm ifix}(T)$ is given in \eqref{e:eq1} and we deduce that ${\rm ifix}(T) >  n^{1/2}$ if and only if $g(p) > 1$, where
\[
g(p):=\frac{2^{(p-1)/2}\left(\frac{1}{2}(p-1)\right)!}{(p-1)(p-2)!^{1/2}}.
\]
Since $g(p+2) < \left(\frac{p}{p-1}\right)^{1/2}g(p+2) = g(p)$ and $g(5)<1$, it follows that ${\rm ifix}(T) \leqs  n^{1/2}$ for every prime $p$ with $p \equiv 1 \imod{4}$. 
\end{rem}

\begin{prop}\label{p:aff2}
If $d \geqs 2$ then either ${\rm ifix}(T) > n^{4/9}$ or $(d,p,{\rm ifix}(T),n) = (2,3,8,840)$.
\end{prop}

\begin{proof}
First assume $d=2$, so $m = p^2$ and $p$ is odd. If $p = 3$ or $5$ then the result follows from Proposition \ref{p:alt}, so let us assume $p \geqs 7$. As in Definition \ref{d:tk}, set $t = t_2 = (e_1,x_2) \in H_0$. By applying Lemma \ref{l:aff} we deduce that 
\[
|t^T\cap H_0| \geqs |t^{H_0}| = p^2,\;\; |t^T| = \frac{(p^2)!}{2^{(p^2-1)/2}(\frac{1}{2}(p^2-1))!}
\]
and thus ${\rm ifix}(T)>n^{4/9}$ if $f(p)>1$, where
\[
f(p):= \frac{2^{(p^2-1)/2}p^{1/3}\left(\frac{1}{2}(p^2-1)\right)!}{(p-1)^{5/9}(p^2-1)^{5/9}(p^2)!^{4/9}}.
\]

We claim that $f(p+2)>f(p)$. To see this, set $k = (p^2+4p+3)/2$ and observe that 
\[
\frac{f(p+2)}{f(p)} = \a \cdot 2^{2p+1}\frac{k!}{(k-2p-2)!}\left(\frac{(2k-4p-3)!}{(2k+1)!}\right)^{4/9}
\]
where
\[
\a = 2\left(\frac{p-1}{p+1}\right)^{5/9}\left(\frac{p+2}{p}\right)^{1/3}\left(\frac{p^2-1}{2k}\right)^{5/9}>1.
\]
By taking logarithms and using the bound $-\frac{x}{1-x} < \log (1-x) < -x$ for all $0<x<1$, it is straightforward to show that 
\begin{equation}\label{e:neww}
a^be^{-\frac{b(b-1)}{2(a-b)}} \leqs \frac{a!}{(a-b)!}\leqs a^be^{-\frac{b(b-1)}{2a}}
\end{equation}
for all integers $1 \leqs b < a$. This implies that 
\[
\frac{f(p+2)}{f(p)}  > \frac{1}{2}e^{\b}\left(\frac{2k}{(2k+1)^{8/9}}\right)^{2p+2},
\]
where
\[
\b = \frac{4}{9}\cdot \frac{(4p+4)(4p+3)}{4k+2} - \frac{(2p+2)(2p+1)}{2(k-2p-2)}.
\]
One checks that this lower bound is minimal when $p=7$, which gives $f(p+2)>f(p)$ as claimed. Moreover, since $f(7)>1$, we conclude that ${\rm ifix}(T)>n^{4/9}$. 

Now assume $d \geqs 3$. If $p=2$ and $d \leqs 6$, then a {\sc Magma} calculation gives ${\rm ifix}(T)>n^{4/9}$. Similarly, one can check that the same conclusion holds if $p=3$ and $d \leqs 4$. In order to establish the desired bound in the remaining cases, 
set $t = t_2 \in H_0$ and note that $t$ has cycle-shape $(2^{p^{d-2}(p^2-1)/2},1^{p^{d-2}})$ by Lemma \ref{l:aff}. Now $|t^T \cap H_0| \geqs 1$ and $|{\rm GL}_{d}(p)|<p^{d^2}$, so $n > (p^d)!p^{-d(d+1)}$ and it follows that ${\rm ifix}(t)>n^{4/9}$ if $g(d,p)>1$, where
\[
g(d,p) := \frac{\left(\frac{1}{2}p^{d-2}(p^2-1)\right)!\left(p^{d-2}\right)!2^{p^{d-2}(p^2-1)/2}}{p^{5d(d+1)/9}\left(p^d\right)!^{4/9}}.
\]

If $d$ is fixed, then by arguing as above one can show that $g(d,p)$ is an increasing function in $p$. Similarly, one checks that if $p$ is fixed, then $g(d,p)$ is increasing as a function of $d$ (here we are assuming that $d \geqs 7$ if $p=2$ and $d \geqs 5$ if $p=3$, which is valid in view of the above remarks). Therefore, for $p \geqs 5$ we have $g(d,p) \geqs g(3,5)>1$. Similarly, if $p=3$ and $d \geqs 5$ then $g(d,p) \geqs g(5,3)>1$ and for $p=2$ with $d \geqs 7$ we get $g(d,p) \geqs g(7,2)>1$. We conclude that ${\rm ifix}(T)> n^{4/9}$ if $d \geqs 3$ and the proof of the proposition is complete. 
\end{proof}

\subsection{Product-type subgroups}\label{ss:prod}

Now assume $H$ is a product-type subgroup of $G$, so $H = (S_k \wr S_r) \cap G$ and $m = k^r$, where $k \geqs 5$ and $r \geqs 2$. Set $\Gamma = \{1, \ldots, k\}$ and note that the embedding of $H$ in $G$ arises from the product action of $H$ on the Cartesian product $\Gamma^r$. That is, for every $(x_1,\ldots,x_r)\s \in H$ and $(\gamma_1, \ldots, \gamma_r) \in \Gamma^r$ we have
\[
(\gamma_1, \ldots, \gamma_r)^{(x_1,\ldots,x_r)\s} = \left(\gamma_1^{x_1}, \ldots, \gamma_r^{x_r}\right)^{\s} = \left(\gamma_{1^{\s^{-1}}}^{x_{1^{\s^{-1}}}},\ldots, \gamma_{r^{\s^{-1}}}^{x_{r^{\s^{-1}}}}\right).
\]
In particular, let us observe that
\[
n \geqs \frac{(k^r)!}{2(k!)^rr!}.
\]

\begin{prop}\label{p:prod}
If $H$ is a product-type subgroup of $G$, then ${\rm ifix}(T)>n^{4/9}$.
\end{prop}

\begin{proof}
Fix the involution $t = (t_1, 1, \ldots, 1) \in (A_k)^r < H_0$, where $t_1 = (1,2)(3,4) \in A_k$. By considering the action of $H$ on $\Gamma^r$, it is easy to see that $t$ has exactly $(k-4)k^{r-1}$ fixed points and so it has cycle-shape $(2^{2k^{r-1}},1^{(k-4)k^{r-1}})$ as an element of $T$. Therefore,
\[
|t^T| = \frac{(k^r)!}{2^{2k^{r-1}}(2k^{r-1})!((k-4)k^{r-1})!}
\]
and using the trivial bound $|t^T \cap H_0| \geqs 1$ we deduce that ${\rm fix}(t)>n^{4/9}$ if $f(k,r)>1$, where
\[
f(k,r):=\frac{2^{18k^{r-1}-5}(2k^{r-1})!^9((k-4)k^{r-1})!^9}{(k^r)!^4(k!)^{5r}(r!)^5}.
\]

A routine calculation shows that if $k$ is fixed, then $f(k,r)$ is an increasing function in $r$, so we may assume $r=2$. If $k = 5$ then $m=25$ and so this case was handled in Proposition \ref{p:alt}. Similarly, if $k = 6$ then an easy {\sc Magma} computation shows that ${\rm ifix}(T)>n^{4/9}$. Finally, if $k \geqs 7$ then  
\[
\frac{f(k+1,2)}{f(k,2)} = 2^{27}\frac{(2k+1)^9}{k+1}\left(\frac{(k^2)!}{(k^2+2k+1)!}\right)^4\left(\frac{(k^2-2k-3)!}{(k^2-4k)!}\right)^9
\]
and by applying the bounds in \eqref{e:neww} it is straightforward to show that this ratio is greater than $1$. In particular, $f(k,2)$ is an increasing function in $k$ and the result follows since $f(7,2)>1$. 
\end{proof}

\subsection{Diagonal-type subgroups}\label{ss:diag}

Here $H = (S^k.({\rm Out}(S) \times S_k)) \cap G$ and $m = |S|^{k-1}$, where $k \geqs 2$ and $S$ is a non-abelian finite simple group. The embedding of $H$ in $G$ is afforded by a natural (faithful) action of $H$ on the set of cosets of the diagonal subgroup $\{(s, \ldots, s) \, : \, s \in S\}$ of $S^k$.

\begin{prop}\label{p:diag}
If $H$ is a diagonal-type subgroup of $G$, then ${\rm ifix}(T)>n^{4/9}$.
\end{prop}

\begin{proof}
First assume $m<200$, so $k=2$ and $S$ is isomorphic to $A_5$ or ${\rm L}_{2}(7)$. In both cases, we can use the database of primitive groups in {\sc Magma} to construct $H$ as a subgroup of $S_m$ and then it is a routine computation to check that $5-5\a-9\b>0$ for constants $\a$ and $\b$ such that  $|H_0| \leqs |T|^{\a}$ and $|t^T| \leqs |T|^{\b}$ for every involution $t \in H_0$. Therefore, ${\rm ifix}(T)>n^{4/9}$ by Lemma \ref{l:basic}.

For the remainder, we may assume $m \geqs 200$. We claim that $|H_0|^{100}<|T|$ and thus ${\rm ifix}(T)>n^{4/9}$ by Lemma \ref{l:cor}. To see this, let us first observe that $|{\rm Out}(S)| \leqs |S|/30$ by \cite[Lemma 2.2]{Quick}, so 
$|H_0| \leqs \frac{1}{30}\ell^{k+1}k!$ where $\ell = |S|$. It follows that $|H_0|^{100}<|T|$ if $f(k,\ell)>1$, where
\[
f(k,\ell):= \frac{1}{2}\left(\frac{30}{\ell^{k+1}k!}\right)^{100}(\ell^{k-1})!
\]

If $k=2$ then $m = \ell$ and the condition $m \geqs 200$ implies that $\ell \geqs 360$ since $A_6$ is the smallest non-abelian simple group with order at least $200$ (up to isomorphism). For $\ell \geqs 360$ we have
\[
\frac{f(2,\ell+1)}{f(2,\ell)} = \left(\frac{\ell}{\ell+1}\right)^{300}(\ell+1) > 1,
\]
so $f(2,\ell)$ is increasing as a function of $\ell$ and we have $f(2,360)>1$. Similarly, if $k \geqs 3$ then $\ell \geqs 60$ and $f(k,\ell)$ is an increasing function in both $k$ and $\ell$. The result now follows since $f(3,60)>1$.
\end{proof}

\subsection{Almost simple subgroups}\label{ss:as}

To complete the proof of Theorem \ref{t:main} for symmetric and alternating groups, we may assume that $T = A_m$ with $m>25$ and $H$ is an almost simple subgroup acting primitively on $\Gamma = \{1, \ldots, m\}$. We will write  $S$ to denote the socle of $H$ (note that $S \ne T$ since $H$ is a core-free subgroup of $G$).

First we handle the low degree groups with $m \leqs 600$. 

\begin{prop}\label{p:as1}
If $25<m \leqs 600$ then ${\rm ifix}(T)>n^{4/9}$.
\end{prop}

\begin{proof}
To construct $H$ as a subgroup of $G$ we use the database of primitive groups in {\sc Magma}, via the command
\[
\texttt{PrimitiveGroups([26..600] : Filter:="AlmostSimple")}.
\]
Once we have removed the groups with $S = T$, we are left with $766$ cases to consider. Define
\[
\a(J) = \max\left\{\frac{|t^{J}|}{|t^T|} \,:\, \mbox{$t \in J$ is an involution}\right\}
\]
for each subgroup $J$ of $H_0$. Given a specific subgroup $J$, we can compute $\a(J)$ by finding a set of representatives for the conjugacy classes of involutions in $J$ and then for each representative $t$ we compute the number of fixed points of $t$ on $\{1, \ldots, m\}$, which allows us to calculate $|t^T|$. Note that ${\rm ifix}(T) \geqs \a(J)n$.

For $m \leqs 60$ it is easy to check that $\a(H_0)>n^{-5/9}$ and thus ${\rm ifix}(T) > n^{4/9}$. Similarly, if $60<m\leqs 600$ and $P$ is a Sylow $2$-subgroup of $H_0$, then $\a(P)>n^{-5/9}$ and the result follows (this approach avoids the problem of computing a set of conjugacy class representatives in $H_0$, which can be expensive in terms of time and memory).
\end{proof}

For the remainder, we may assume $m>600$. Our basic aim is to establish the bound
\begin{equation}\label{e:T}
|H_0|^{100}<|T|
\end{equation}
whenever possible, noting that this gives ${\rm ifix}(T)>n^{4/9}$ via Lemma \ref{l:cor}. To do this, it will be convenient to make a distinction between the cases where $H$ is standard or non-standard, according to the following definition.

\begin{defn}\label{d:std}
Let $H \leqs {\rm Sym}(\Gamma)$ be an almost simple primitive group with socle $S$ and point stabilizer $K$. Then $H$ is \emph{standard} if one of the following holds:
\begin{itemize}\addtolength{\itemsep}{0.2\baselineskip}
\item[{\rm (i)}] $S=A_{k}$ is an alternating group and $\Gamma$ is a set of subsets or partitions of $\{1,\ldots, k\}$.
\item[{\rm (ii)}] $S$ is a classical group with natural module $V$ and $\Gamma$ is a set of subspaces (or pairs of subspaces) of $V$.
\item[{\rm (iii)}] $S = {\rm Sp}_{2d}(q)$, $q$ is even and $K \cap S = {\rm O}_{2d}^{\pm}(q)$.
\end{itemize}
In all other cases, $H$ is \emph{non-standard}. 
\end{defn}

This definition facilitates the statement of the following key result of Liebeck and Saxl (see \cite[Proposition 2]{LSa2}).

\begin{prop}\label{p:bd}
Let $H \leqs {\rm Sym}(\Gamma)$ be a non-standard almost simple primitive group of degree $m \geqs 25$. Then $|H|<m^5$.
\end{prop}

With this proposition in hand, we can very quickly reduce the problem to the groups where $H$ is standard.

\begin{prop}\label{p:as2}
If $m>600$ and $H$ is non-standard, then ${\rm ifix}(T)>n^{4/9}$.
\end{prop}

\begin{proof}
Here $|H|<m^5$ by Proposition \ref{p:bd} and one can check that
$2m^{500} < m!$ (since $m>600$). Therefore \eqref{e:T} holds and the result follows.
\end{proof}

\begin{prop}\label{p:as3}
If $m>600$, $H$ is standard and $S$ is alternating, then ${\rm ifix}(T)>n^{4/9}$.
\end{prop}

\begin{proof}
Write $S=A_k$ and first assume that the embedding of $H$ in $G$ is afforded by the action of $H$ on the set of $\ell$-element subsets of $\{1, \ldots, k\}$, so $m = \binom{k}{\ell}$ and $2 \leqs \ell < k/2$. Note that $k \geqs 12$ since $m>600$. Now $m \geqs \binom{k}{2} = \frac{1}{2}k(k-1)$ and it is straightforward to check that 
\[
|H_0|^{100} \leqs (k!)^{100} < \frac{1}{2}\left(\frac{1}{2}k(k-1)\right)! \leqs |T|
\]
for all $k \geqs 98$. Similarly, if $k \leqs 97$ and $\ell \geqs 3$ then $m \geqs \binom{k}{3}$ and one checks that \eqref{e:T} holds, so we may assume that $\ell=2$, $36 \leqs k \leqs 97$ and $m = \frac{1}{2}k(k-1)$. Here we compute
\begin{equation}\label{e:al}
\a = \frac{\log k!}{\log |T|},\;\; \b = \frac{\log \gamma}{\log |T|},
\end{equation}
where 
\[
\gamma = \max\left\{\frac{m!}{2^{2j}(2j)!(m-4j)!}\,:\, 1 \leqs j \leqs m/4\right\}
\]
is the size of the largest conjugacy class of involutions in $T$. One checks that $5-5\a-9\b>0$ in each case, whence ${\rm ifix}(T)>n^{4/9}$ by Lemma \ref{l:basic}.

Now assume that the embedding of $H$ corresponds to the action on the set of partitions of $\{1, \ldots, k\}$ into $r$ subsets of size $\ell$, where $1<\ell<k$. Here $m = \frac{k!}{(\ell!)^rr!}$ and the condition $m>600$ implies that $k \geqs 10$. It is easy to check that $m \geqs \binom{k}{4}$ and by arguing as in the previous paragraph we deduce that \eqref{e:T} holds if $k \geqs 12$. The same bound also holds when $k=10$ since $r=5$, $\ell=2$ and $m=945$. 
\end{proof}

In order to complete the proof of Theorem \ref{t:main} for $T = A_m$ we may assume that $m>600$ and $H$ is an almost simple classical group over $\mathbb{F}_q$ with socle $S$. Let $V$ be the natural module for $S$ and set $\ell = \dim V$. In view of Proposition \ref{p:as2}, we may also assume that $H \leqs {\rm Sym}(\Gamma)$ is a \emph{standard} group, which means that $\Gamma$ is either a set of subspaces (or pairs of subspaces) of $V$, or $S = {\rm Sp}_{\ell}(q)$, $q$ is even and $\Gamma$ is the set of cosets of a subgroup ${\rm O}_{\ell}^{\pm}(q)$ of $S$ (see Definition \ref{d:std}). Let $K$ be a point stabilizer for the action of $H$ on $\Gamma$, so $m = |H:K|$.

\begin{rem}\label{r:cond}
Due to the existence of a number of exceptional isomorphisms among the low dimensional classical groups, we may assume that $S$ is one of the following:
\[
{\rm L}_{\ell}(q),\, \ell \geqs 2; \; {\rm U}_{\ell}(q), \, \ell \geqs 3; \; {\rm PSp}_{\ell}(q), \, \ell \geqs 4;\; {\rm P\O}_{\ell}^{\e}(q),\, \ell \geqs 7.
\]
In addition, in view of the isomorphisms
\[
{\rm L}_{2}(4) \cong {\rm L}_{2}(5) \cong A_5, \; {\rm L}_{2}(9) \cong {\rm PSp}_{4}(2)' \cong A_6,\; 
{\rm L}_{3}(2) \cong {\rm L}_{2}(7),\; {\rm L}_{4}(2) \cong A_8,\; {\rm PSp}_{4}(3) \cong {\rm U}_{4}(2)
\]
(see \cite[Proposition 2.9.1]{KL}), we may assume that 
\[
S \ne {\rm L}_{2}(4), \, {\rm L}_{2}(5), \, {\rm L}_{2}(9), \, {\rm L}_{3}(2), \, {\rm L}_{4}(2), \, {\rm PSp}_{4}(2)', \, {\rm PSp}_{4}(3).
\]

\end{rem}

\begin{prop}\label{p:as4}
If $m>600$, $H$ is standard and $S$ is classical, then ${\rm ifix}(T)>n^{4/9}$.
\end{prop}

\begin{proof}
We adopt the set-up introduced above, including the conditions on $S$ presented in Remark \ref{r:cond}.  Write $q=p^f$, where $p$ is a prime. We will prove that \eqref{e:T} holds unless $(H,m) = ({\rm L}_{10}(2), 2^{10}-1)$.

Since $m = |S:S \cap K|$ it follows that $m \geqs P(S)$, where $P(S)$ is the minimal degree of a nontrivial permutation representation of $S$. The minimal degrees are presented in \cite[Table 4]{GMPS} (which corrects a couple of slight errors in \cite[Table 5.2.A]{KL}) and by inspection we deduce that $m > q^{\ell-2}$. Similarly, the order of ${\rm Aut}(S)$ is recorded in \cite[Table 5.1.A]{KL} and it is easy to see that $|H| \leqs |{\rm Aut}(S)|<2q^{\ell^2}$. 

If $S = {\rm L}_{2}(q)$ then $m > \max\{600,q\}$, $|H| \leqs q(q^2-1) \log_pq$ and it is routine to verify the bound in \eqref{e:T}. Similarly, if $\ell=3$ then $S = {\rm L}_{3}^{\e}(q)$, $m > \max\{600,q^2+q\}$, 
\[
|H| \leqs |{\rm Aut}({\rm U}_{3}(q))| = 2q^3(q^2-1)(q^3+1)\log_pq
\]
and we quickly deduce that \eqref{e:T} holds.  

Now assume $\ell \geqs 4$. If $q \geqs 31$ then one checks that
\begin{equation}\label{e:ell}
2^{101}q^{100\ell^2} < (q^{\ell-2})!
\end{equation}
and this establishes the bound in \eqref{e:T}. More precisely, if $\ell \geqs 5$ then the inequality in \eqref{e:ell} is satisfied unless $(\ell,q)$ is one of the following:
\begin{equation}\label{e:cases}
\ell = 5, \, q \leqs 9; \; \ell=6, \, q \leqs 5; \; \ell=7, \, q \leqs 4; \; \ell=8, \, q \leqs 3; \; \ell=9,10,11,12, \, q = 2.
\end{equation}

Suppose $\ell=4$ and $q \leqs 29$. If $S = {\rm PSp}_{4}(q)$ then $q \geqs 4$ (see Remark \ref{r:cond}), $m > \max\{600,q^2\}$ and
\[
|H| \leqs |{\rm Aut}({\rm PSp}_{4}(q))| \leqs 2q^4(q^2-1)(q^4-1)\log_pq,
\]
which implies that \eqref{e:T} holds. Now assume $S = {\rm L}_{4}^{\e}(q)$, so 
\[
|H| \leqs |{\rm Aut}({\rm U}_{4}(q))| \leqs 2q^6(q^2-1)(q^3+1)(q^4-1)\log_pq.
\]
If $m> \max\{600,q^4\}$ then \eqref{e:T} holds, so let us assume $600 <  m \leqs q^4$. By inspecting \cite[Table 4.1.2]{BG}, which records the degree of every standard classical group, we deduce that $S = {\rm L}_{4}(q)$ and $m=(q^4-1)/(q-1)$ is the only possibility, so $q \geqs 9$ and we get $|{\rm Aut}(S)|^{100}<|T|$.

Very similar reasoning establishes the bound in \eqref{e:T} for all the remaining cases in \eqref{e:cases} with $\ell \leqs 9$, so to complete the proof we may assume that $\ell \in \{10,11,12\}$ and $q=2$. If $\ell \in \{11,12\}$ and $S = {\rm L}_{\ell}^{\e}(2)$ then $m \geqs 2^{\ell}-1$ and we deduce that \eqref{e:T} holds. Similarly, if $\ell=12$ and $S \ne {\rm L}_{12}^{\e}(2)$ then the bound $m > 2^{10}$ is sufficient. Finally, let us assume 
$\ell=10$. If $S \ne {\rm L}_{10}^{\e}(2)$ then $|H| \leqs |{\rm Sp}_{10}(2)|$ and one checks that the condition $m>600$ implies that $m \geqs 2^{10}-1$, which allows us to verify the bound in \eqref{e:T}. Now assume $S = {\rm L}_{10}^{\e}(2)$, so $|H| \leqs 2|{\rm U}_{10}(2)|$. If $m>2^{11}$ then $|H|^{100}<|T|$, so let us assume $m \leqs 2^{11}$, in which case $H = {\rm L}_{10}(2)$ and $m = 2^{10}-1$ (so $\Gamma$ is the set $1$-dimensional subspaces of $V$). Here $|H|^{100} > |T|$, but if we define $\a = \log |H| /\log |T|$ and $\b$ as in \eqref{e:al}, then it is easy to check that $5-5\a-9\b>0$ and thus ${\rm ifix}(T)>n^{4/9}$ by Lemma \ref{l:basic}. 
\end{proof}

\section{Sporadic groups}\label{s:spor}

In this final section we complete the proof of Theorem \ref{t:main} by handling the groups with socle a sporadic simple group. Our first result quickly reduces the problem to the Baby Monster and Monster (denoted by $\mathbb{B}$ and $\mathbb{M}$, respectively).

\begin{prop}\label{l:spor1}
The conclusion to Theorem \ref{t:main} holds if $T \ne \mathbb{B}, \mathbb{M}$ is a sporadic group.
\end{prop}

\begin{proof}
This is an easy computation using the \textsf{GAP} Character Table Library \cite{GAPCTL}. In each case, the character table of $G$ is available in \cite{GAPCTL} and we use the \texttt{Maxes} function to access the character table of the maximal subgroup $H$. In addition, \cite{GAPCTL} stores the fusion map from $H$-classes to $G$-classes, which allows us to compute ${\rm fix}(t)$ via \eqref{e:fix} for all $t \in G$. In particular, we can compute ${\rm ifix}(T)$ precisely and the result follows.
\end{proof}

To complete the proof of Theorem \ref{t:main}, we may assume that $T = \mathbb{B}$ or $\mathbb{M}$. In both cases, we claim that ${\rm ifix}(T)>n^{4/9}$.

\begin{prop}\label{l:spor2}
The conclusion to Theorem \ref{t:main} holds if $T  = \mathbb{B}$.
\end{prop}

\begin{proof}
Here $G = T$ is the Baby Monster and we proceed as in the proof of the previous proposition, noting that the character tables of $G$ and $H$ are available in \cite{GAPCTL} (as before, we use the \texttt{Maxes} function to access the character table of $H$). In addition, in all but one case, the fusion map from $H$-classes to $G$-classes is also stored and this reduces the analysis to the case $H = (2^2 \times F_4(2)).2$. Here we use the function \texttt{PossibleClassFusions} to determine a set of candidate fusion maps (there are $64$ such maps in total) and for each possibility one checks that ${\rm ifix}(T) = 1609085288448 > n^{4/9}$. 
\end{proof}

\begin{prop}\label{l:spor3}
The conclusion to Theorem \ref{t:main} holds if $T  = \mathbb{M}$.
\end{prop}

\begin{proof}
Let $G = T = \mathbb{M}$ be the Monster. By inspecting the \textsc{Atlas} \cite{Atlas}, we see that $G$ has two conjugacy classes of involutions, labelled \texttt{2A} and \texttt{2B}, where
\[
|\texttt{2A}| = 97239461142009186000,\;\; |\texttt{2B}| = 5791748068511982636944259375.
\]
As discussed in \cite{Wil}, $G$ has $44$ known conjugacy classes of maximal subgroups and any additional maximal subgroup is almost simple with socle one of ${\rm L}_{2}(8)$, ${\rm L}_{2}(13)$, ${\rm L}_{2}(16)$ or ${\rm U}_{3}(4)$. Let us define the following three collections of known maximal subgroups of $G$:
\begin{align*}
\mathcal{A} & = \{2^{10+16}.\O_{10}^{+}(2),  2^{2+11+22}.({\rm M}_{24} \times S_3), 2^{5+10+20}.(S_3 \times {\rm L}_{5}(2)), 2^{3+6+12+18}.({\rm L}_{3}(2)  \times 3S_6) \} \\
\mathcal{B} & = \{ 3^8.{\rm P\O}_{8}^{-}(3).2, (3^2{:}2 \times {\rm P\O}_{8}^{+}(3)).S_4, 3^{2+5+10}.({\rm M}_{11} \times 2S_4), 3^{3+2+6+6}.({\rm L}_{3}(3) \times {\rm SD}_{16})\} \\
\mathcal{C} & = \{({\rm L}_{2}(11) \times {\rm L}_{2}(11)){:}4, 11^2{:}(5 \times 2A_5), 7^2{:}{\rm SL}_{2}(7), {\rm L}_{2}(29){:}2, {\rm L}_{2}(19){:}2\}
\end{align*}

First assume $H$ belongs to one of the $44$ known classes of maximal subgroups. If $H$ is not contained in $\mathcal{A}$, $\mathcal{B}$ or $\mathcal{C}$, then we use the function \texttt{NamesOfFusionSources} to access the character table of $H$ in \textsf{GAP} and in each case we can work with the stored fusion map from $H$-classes to $G$-classes. This allows us to compute ${\rm ifix}(T)$ as in the proof of Proposition \ref{l:spor1} and it is straightforward to verify the desired bound.

Now assume $H$ is one of the subgroups in $\mathcal{C}$. Here the character table of $H$ is available in \textsf{GAP}, but the fusion map is not stored. So in these cases we proceed as in the proof of Proposition \ref{l:spor2}, using the function \texttt{PossibleClassFusions}. In each case, we find that ${\rm ifix}(T)$ is independent of the choice of fusion map and we calculate that ${\rm ifix}(T)>n^{4/9}$.

Next suppose $H \in \mathcal{A} \cup \mathcal{B}$. If $H \in \mathcal{A}$ then \cite[Proposition 3.9]{BOW} gives $|t^G \cap H|$, where $t$ is contained in the \texttt{2A} class of involutions in $G$. This allows us to compute ${\rm fix}(t)$ precisely and we deduce that ${\rm ifix}(T)>n^{4/9}$. Now assume $H \in \mathcal{B}$ and let $\a$ be the size of the largest conjugacy class of involutions in $H$. We use {\sc Magma} to compute $\a$, working with a representation of $H$ given in \cite{WebAt}. For $H = 3^8.{\rm P\O}_{8}^{-}(3).2$, this is a matrix representation of dimension $204$ over $\mathbb{F}_3$ and we use \texttt{LMGClasses} to compute a set of conjugacy class representatives; in the remaining cases, we work with a permutation representation of degree less than $10^5$. Now, if $\b$ is the size of the \texttt{2B} class of involutions in $G$ (see above),  then ${\rm ifix}(T) \geqs n\a\b^{-1}$ and in each case it is easy to check that this lower bound is greater than $n^{4/9}$. For example, if $H = 3^8.{\rm P\O}_{8}^{-}(3).2$ then $\a = 1982806371$ and the above bound yields ${\rm ifix}(T) > n^{4/9}$.

To complete the proof of the proposition, we may assume $H$ is an almost simple maximal subgroup with socle $S = {\rm L}_{2}(8)$, ${\rm L}_{2}(13)$, ${\rm L}_{2}(16)$ or ${\rm U}_{3}(4)$. Let $\a$ be the size of the largest class of involutions in $S$ and define $\b = |\texttt{2B}|$ as above. Then one checks that
\[
\left(\frac{|T|}{|{\rm Aut}(S)|}\right)^{5/9} > \frac{\b}{\a}
\]
and since $H \leqs {\rm Aut}(S)$, we immediately deduce that ${\rm ifix}(T)>n^{4/9}$ as required.
\end{proof}

\vs

This completes the proof of Theorem \ref{t:main}.

\end{document}